\documentclass[12pt]{amsart}

\title[Set-theoretic mereology]{Set-theoretic mereology}

\author[Hamkins]{Joel David Hamkins}
 \address[J.~D.~Hamkins]
         {Mathematics, Philosophy, Computer Science, The Graduate Center of The City University of New York,
         365 Fifth Avenue, New York, NY 10016 \&
         Mathematics, College of Staten Island of CUNY, Staten Island, NY 10314}
 \email{jhamkins@gc.cuny.edu}
 \urladdr{http://jdh.hamkins.org}
 \thanks{This work was initiated at the conference on Computability Theory and the Foundations of Mathematics (CTFM) at the Tokyo Institute of Technology, held in honor of Kazuyuki Tanaka in September 2015. The first author is extremely grateful to his Japanese hosts, including Toshimichi Usuba, for supporting his research visit there. The research of the first author has also been supported in part by the Simons Foundation (grant 209252). Commentary concerning this paper can be made at \href{http://jdh.hamkins.org/set-theoretic-mereology}{jdh.hamkins.org/set-theoretic-mereology}.}

\author[Kikuchi]{Makoto Kikuchi}
 \address[M.~Kikuchi]
         {Graduate School of System Informatics, Kobe University, Rokkodai, Nada, Kobe 657-8501, Japan}
 \email{mkikuchi@kobe-u.ac.jp}
 \urladdr{http://www2.kobe-u.ac.jp/~mkikuchi/index-e.html}

\newtheorem{theorem}{Theorem}

\newtheorem*{maintheorem*}{Main Theorem}
\newtheorem*{maintheorems*}{Main Theorems}
\newtheorem{corollary}[theorem]{Corollary}

\newtheorem{question}[theorem]{Question}

\newtheorem*{questions*}{Questions}
\newtheorem*{mainquestion*}{Main Question} 
\newtheorem*{openquestion*}{Open Question} 

\newcommand{\QED}{\end{proof}}

\def\proclaim[#1]{{\bf #1}}
\def\BF#1.{{\bf #1.}}

\newcommand\Ersov{Er\v sov}

\newcommand{\Godel}{G\"odel}

\newcommand{\Lowenheim}{L\"owenheim}

\newcommand{\N}{{\mathbb N}}

\newcommand{\of}{\subseteq}

\newcommand{\ofneq}{\subsetneq}

\newcommand{\singleton}[1]{\left\{{#1}\right\}}

\newcommand{\elesub}{\prec}

\newcommand{\Con}{\mathop{{\rm Con}}}
\newcommand{\image}{\mathbin{\hbox{\tt\char'42}}}

\newcommand{\satisfies}{\models}

\newcommand{\union}{\cup}

\newcommand{\intersect}{\cap}
\newcommand{\Intersect}{\bigcap}

\newcommand{\smalllt}{\mathrel{\mathchoice{\raise2pt\hbox{$\scriptstyle<$}}{\raise1pt\hbox{$\scriptstyle<$}}{\raise0pt\hbox{$\scriptscriptstyle<$}}{\scriptscriptstyle<}}}
\newcommand{\smallleq}{\mathrel{\mathchoice{\raise2pt\hbox{$\scriptstyle\leq$}}{\raise1pt\hbox{$\scriptstyle\leq$}}{\raise1pt\hbox{$\scriptscriptstyle\leq$}}{\scriptscriptstyle\leq}}}

\newcommand{\boolval}[1]{\mathopen{\lbrack\!\lbrack}\,#1\,\mathclose{\rbrack\!\rbrack}}
\def\[#1]{\boolval{#1}}
\newbox\gnBoxA
\newdimen\gnCornerHgt
\setbox\gnBoxA=\hbox{\tiny$\ulcorner$}
\global\gnCornerHgt=\ht\gnBoxA
\newdimen\gnArgHgt
\def\gcode #1{%
\setbox\gnBoxA=\hbox{$#1$}%
\gnArgHgt=\ht\gnBoxA%
\ifnum     \gnArgHgt<\gnCornerHgt \gnArgHgt=0pt%
\else \advance \gnArgHgt by -\gnCornerHgt%
\fi \raise\gnArgHgt\hbox{\tiny$\ulcorner$} \box\gnBoxA %
\raise\gnArgHgt\hbox{\tiny$\urcorner$}}
\newcommand{\UnderTilde}[1]{{\setbox1=\hbox{$#1$}\baselineskip=0pt\vtop{\hbox{$#1$}\hbox to\wd1{\hfil$\sim$\hfil}}}{}}
\newcommand{\Undertilde}[1]{{\setbox1=\hbox{$#1$}\baselineskip=0pt\vtop{\hbox{$#1$}\hbox to\wd1{\hfil$\scriptstyle\sim$\hfil}}}{}}
\newcommand{\undertilde}[1]{{\setbox1=\hbox{$#1$}\baselineskip=0pt\vtop{\hbox{$#1$}\hbox to\wd1{\hfil$\scriptscriptstyle\sim$\hfil}}}{}}
\newcommand{\UnderdTilde}[1]{{\setbox1=\hbox{$#1$}\baselineskip=0pt\vtop{\hbox{$#1$}\hbox to\wd1{\hfil$\approx$\hfil}}}{}}
\newcommand{\Underdtilde}[1]{{\setbox1=\hbox{$#1$}\baselineskip=0pt\vtop{\hbox{$#1$}\hbox to\wd1{\hfil\scriptsize$\approx$\hfil}}}{}}

\renewcommand{\implies}{\mathrel{\rightarrow}}

\renewcommand{\iff}{\mathrel{\longleftrightarrow}}

\def\<#1>{\left\langle#1\right\rangle}

\newcommand{\ZFC}{{\rm ZFC}}

\newcommand{\HF}{{\rm HF}}

\newcommand{\cell}[1]{\boxit{\hbox to 17pt{\strut\hfil$#1$\hfil}}}
\newcommand{\head}[2]{\lower2pt\vbox{\hbox{\strut\footnotesize\it\hskip3pt#2}\boxit{\cell#1}}}
\newcommand{\boxit}[1]{\setbox4=\hbox{\kern2pt#1\kern2pt}\hbox{\vrule\vbox{\hrule\kern2pt\box4\kern2pt\hrule}\vrule}}
\newcommand{\Col}[3]{\hbox{\vbox{\baselineskip=0pt\parskip=0pt\cell#1\cell#2\cell#3}}}
\newcommand{\tapenames}{\raise 5pt\vbox to .7in{\hbox to .8in{\it\hfill input: \strut}\vfill\hbox to
.8in{\it\hfill scratch: \strut}\vfill\hbox to .8in{\it\hfill output: \strut}}}
\newcommand{\Head}[4]{\lower2pt\vbox{\hbox to25pt{\strut\footnotesize\it\hfill#4\hfill}\boxit{\Col#1#2#3}}}
\newcommand{\Dots}{\raise 5pt\vbox to .7in{\hbox{\ $\cdots$\strut}\vfill\hbox{\ $\cdots$\strut}\vfill\hbox{\
$\cdots$\strut}}}
\newcommand{\df}{\it} 
\usepackage[hidelinks]{hyperref}
\usepackage{latexsym,amsfonts,amsmath,amssymb}
\usepackage{tikz} 
\usetikzlibrary{arrows}
\usepackage{diagrams}\diagramstyle[tight,centredisplay,textflow]

\begin{document}

\begin{abstract}
 We consider a set-theoretic version of mereology based on the inclusion relation~$\of$ and analyze how well it might serve as a foundation of mathematics. After establishing the non-definability of $\in$ from $\of$, we identify the natural axioms for $\of$-based mereology, which constitute a finitely axiomatizable, complete, decidable theory. Ultimately, for these reasons, we conclude that this form of set-theoretic mereology cannot by itself serve as a foundation of mathematics. Meanwhile, augmented forms of set-theoretic mereology, such as that obtained by adding the singleton operator, are foundationally robust.
\end{abstract}

\maketitle

\noindent
In light of the comparative success of membership-based set theory in the foundations of mathematics, since the time of Cantor, Zermelo and Hilbert, a mathematical philosopher naturally wonders whether one might find a similar success for mereology, based upon a mathematical or set-theoretic parthood relation rather than the element-of relation~$\in$. Can set-theoretic mereology serve as a foundation of mathematics? And what should be the central axioms of set-theoretic mereology?

We should like therefore to consider a mereological perspective in set theory, analyzing how well it might serve as a foundation while identifying the central axioms. Although set theory and mereology, of course, are often seen as being in conflict, what we take as the project here is to develop and investigate, within set theory, a set-theoretic interpretation of mereological ideas. Mereology, by placing its focus on the parthood relation, seems naturally interpreted in set theory by means of the inclusion relation~$\of$, so that one set $x$ is a {\df part} of another $y$, just in case $x$ is a subset of $y$, written $x\of y$. This interpretation agrees with David Lewis's~\cite{Lewis1991:PartsOfClasses} interpretation of set-theoretic mereology in the context of sets and classes (see also~\cite{Hellman2009:MereologyInPhilosophyOfMathematics}), but we restrict our attention to the universe of sets. So in this article we shall consider the formulation of set-theoretic mereology as the study of the structure $\<V,\of>$, which we shall take as the canonical fundamental structure of set-theoretic mereology, where $V$ is the universe of all sets; this is in contrast to the structure $\<V,{\in}>$, usually taken as central in set theory. The questions are: How well does this mereological structure serve as a foundation of mathematics? Can we faithfully interpret the rest of mathematics as taking place in $\<V,\of>$ to the same extent that set theorists have argued (with whatever degree of success) that one may find faithful representations in $\<V,{\in}>$? Can we get by in mathematics with merely the subset relation~$\of$ in place of the membership relation~$\in$?

Ultimately, we shall identify grounds supporting generally negative answers to these questions. On the basis of various mathematical results, our main philosophical thesis will be that the particular understanding of set-theoretic mereology via the inclusion relation~$\of$ cannot adequately serve by itself as a foundation of mathematics. Specifically, theorem~\ref{Theorem.DifferentInSameOf} and corollary~\ref{Corollary.InNotDefinableFromOf} show that~$\in$ is not definable from~$\of$, and we take this to show that one may not interpret membership-based set theory itself within set-theoretic mereology in any straightforward, direct manner. A counterpoint to this is provided by theorem~\ref{Theorem.IfInclusionSameThenInIsIsomorphic}, however, which identifies a weak sense in which~$\of$ may identify~$\in$ up to definable automorphism of the universe. That counterpoint is not decisive, however, in light of question~\ref{Question.NonIsomorphicInButSameSubset} and its resolution by theorem~\ref{Theorem.DifferentTheoriesSameOf}, which shows that set-theoretic mereology does not actually determine the~$\in$-isomorphism class or even the~$\in$-theory of the~$\in$-model in which it arises. For example, we cannot determine in $\of$-based set-theoretic mereology whether the continuum hypothesis holds or fails, whether the axiom of choice holds or fails or whether there are large cardinals or not. Initially, theorem~\ref{Theorem.MereologyIsDecidable} may appear to be a positive result for mereology, since it identifies precisely what are the principles of set-theoretic mereology, considered as the theory of $\<V,{\of}>$. Namely,~$\of$ is an atomic unbounded relatively complemented distributive lattice, and this is a finitely axiomatizable complete theory. So in a sense, this theory simply \emph{is} the theory of $\of$-based set-theoretic mereology. But upon reflection, since every finitely axiomatizable complete theory is decidable, the result actually appears to be devastating for set-theoretic mereology as a foundation of mathematics, because a decidable theory is much too simple to serve as a foundational theory for all mathematics. The full spectrum and complexity of mathematics naturally includes all the instances of many undecidable decision problems and so cannot be founded upon a decidable theory. Finally, corollary~\ref{Corollary.HFElementary} shows that the structure consisting of the hereditarily finite sets under inclusion forms an elementary substructure of the full set-theoretic mereological universe
 $$\<\HF,\of>\elesub\<V,\of>.$$
Consequently, set-theoretic mereology cannot properly treat or even express the various concepts of infinity that arise in mathematics.

Let us briefly clarify the foundational dialectic of this article. We study set-theoretic mereology within set theory itself, studying $\<V,{\of}>$ as it is defined in set theory, by working in \ZFC, for example, although much weaker set theories would suffice for our analysis. Every model $\<M,\in^M>$ of $\ZFC$ gives rise to the associated canonical model $\<M,{\of^M}>$ of set-theoretic mereology. So we make claims about~$\of$ when we are able to prove them in our~$\in$-based set theory. If a mereologist desires instead to work axiomatically purely within set-theoretic mereology itself, then in light of theorem~\ref{Theorem.MereologyIsDecidable} there will be no disagreement on the fundamental mereological truths, if $\of$ is regarded as an atomic unbounded relatively complemented distributive lattice, since this is a complete theory. To be sure, some of the axiomatizations of mereology that have been proposed in the literature do not agree with those axioms, particularly on the issue of atomicity as we discuss in section~\ref{Section.MereologyWithSingltonOperator}; but to the extent that they disagree with those elementary set-theoretic properties of inclusion $\of$, we regard them as concerned with another kind of mereology and not with $\of$-based set-theoretic mereology, the topic on which we are focussing. In our membership-based set theory, meanwhile, we argue that there are substantive mathematical concepts and truths that are not captured by the mereological theory of~$\of$, and for this reason, it does not serve adequately as a foundation of mathematics.

Let us also briefly remark on the distinction between the parthood relation and the proper parthood relation, as some mereologists prefer to take the proper-part relation as fundamental, rather than the parthood relation. In set theory, this distinction amounts to the distinction between the inclusion relation~$\of$, which is reflexive, and the proper-subset relation $\ofneq$, which is irreflexive. These two relations, however, are easily interdefinable in the first-order language of set theory as follows:
 $$x\of y\quad\longleftrightarrow\quad x\ofneq y\vee x=y$$
 $$x\ofneq y\quad\longleftrightarrow\quad x\of y\wedge x\neq y.$$
For our conclusions, therefore, it does not seem to matter which of these relations we consider to be fundamental. In particular,~$\in$ is not definable from either of them, and set-theoretic mereology is decidable, whether one is considering the theory of $\<V,\of>$ or of $\<V,\ofneq>$. Henceforth, therefore, we shall without loss of generality focus on the reflexive relation~$\of$.

Finally, although we argue in this article that the particular formulation of set-theoretic mereology interpreted via the inclusion relation~$\of$ cannot by itself serve adequately as a foundation of mathematics, nevertheless we should like to remark that there may be alternative mereological perspectives in set theory, using a different interpretation of the parthood relation, that do allow it to serve as a suitable foundation of mathematics. For example, theorem~\ref{Theorem.MereologyWithSingletonOperator} shows that if one augments set-theoretic mereology with the singleton operator $a\mapsto\{a\}$, then it becomes interdefinable with~$\in$-based set theory, and therefore just as adequate as~$\in$ in foundations. Meanwhile, there are still other interpretations of mereology in set theory. See \cite{Varzi2015:Mereology} for a general survey of mereology.

\section{Non-definability of~$\in$ from~$\of$}

Our initial task, of course, is to settle the question of whether the two set-theoretic relations~$\in$ and~$\of$ might be definable from one another or otherwise bi-interpretable in set theory. For if the two relations were interdefinable, then we would reasonably see them as fundamentally equivalent in their capacity to serve as a foundation of mathematics, since either would serve as a foundation for the other, and this would settle the entire issue. Furthermore, elementary classical set theory already provides one direction, since~$\of$ is easily definable from~$\in$ in set theory via
 $$x\of y\quad\longleftrightarrow\quad \forall z\,(z\in x\implies z\in y).$$
So the question really is whether one may conversely define~$\in$ from~$\of$. At the CTFM 2015 in Tokyo, the second author specifically asked for a counterexample model:

\begin{question}[Kikuchi]
 Can there be two models of set theory with different membership relations, but the same inclusion relation?
\end{question}

More specifically, he asks for models of set theory $\<V,{\in}>$ and $\<V,{\in^*}>$ on the same underlying universe of sets $V$, with different membership relations $\in\neq\in^*$, but for which the correspondingly defined inclusion relations~$\of$ and $\of^*$, respectively, are identical. The answer is yes, and indeed, every model of set theory has many such alternative membership relations with the same inclusion relation:

\begin{theorem}\label{Theorem.DifferentInSameOf}
In any universe of set theory $\<V,{\in}>$, there is a definable relation $\in^*$, different from~$\in$, such that $\<V,{\in^*}>$ is a model of set theory, in fact isomorphic to the original universe $\<V,{\in}>$, for which the corresponding inclusion relation $$u\subseteq^* v\quad\longleftrightarrow\quad \forall a\, (a\in^* u\to a\in^* v)$$ is identical to the usual inclusion relation $u\subseteq v$.
\end{theorem}

\begin{proof}
The result requires very little about the theory of $\<V,{\in}>$, and even extremely weak set theories suffice. Let $\theta:V\to V$ be any definable non-identity permutation of the universe, such as the permutation that swaps $\emptyset$ and $\{\emptyset\}$ and leaves all other sets unchanged. Let $\tau:u\mapsto \theta\image u=\{\ \theta(a)\mid a\in u\ \}$ be the function determined by pointwise image under $\theta$. Since $\theta$ is bijective and definable, it follows that $\tau$ is also a bijection of $V$ to $V$, since every set is the $\theta$-image of a unique set. Furthermore, $\tau$ is an automorphism of $\<V,\subseteq>$, since $$u\subseteq v\quad\iff\quad\theta\image u\subseteq\theta\image v\quad\iff\quad\tau(u) \subseteq\tau(v).$$
The first author had used this idea a few years ago in~\cite{Hamkins2013.MO144236:IsTheInclusionVersionOfTheKunenInconsistencyTheoremTrue?} in order to prove that there are always many nontrivial $\subseteq$-automorphisms of the set-theoretic universe, as expressed in corollary~\ref{Corollary.NontrivialAutomorphismsOfOf}. Note that since $\tau(\{a\})=\{\theta(a)\}$, it follows that any instance of nontriviality $\theta(a)\neq a$ in $\theta$ leads immediately to an instance of nontriviality in $\tau$.

Using the map $\tau$, we define $a\in^* b\iff\tau(a)\in\tau(b)$. By definition, therefore, and since $\tau$ is bijective, it follows that $\tau$ is an isomorphism of the structures $\<V,\in^*>\cong\<V,\in>$, and so in particular, $\<V,{\in^*}>$ has the same theory as $\<V,{\in}>$, making it just as much a model of set theory as $\<V,{\in}>$ is. Let us show next that $\in^*\neq \in$. Since $\theta$ is nontrivial, there is an~$\in$-minimal set $a$ with $\theta(a)\neq a$ (one can take $a=\emptyset$ for the particular $\theta$ that we provided above). By minimality, $\theta\image a=a$ and so $\tau(a)=a$. But as mentioned, $\tau(\{a\})=\{\theta(a)\}\neq\{a\}$. So we have $a\in\{a\}$, but $\tau(a)=a\notin\{\theta(a)\}=\tau(\{a\})$ and hence $a\notin^*\{a\}$. So the two relations~$\in$ and $\in^*$ are different.

Meanwhile, consider the corresponding subset relation. Specifically, $u\subseteq^* v$ is defined to mean $\forall a\,(a\in^* u\to a\in^* v)$, which holds if and only if $\forall a\, (\tau(a)\in\tau(u)\to \tau(a)\in\tau(v))$; but since $\tau$ is surjective, this holds if and only if $\tau(u)\subseteq \tau(v)$, which as we observed at the beginning of the proof, holds if and only if $u\subseteq v$. So the corresponding subset relations $\subseteq^*$ and $\subseteq$ are identical, as desired.
\end{proof}

\begin{corollary}[\cite{Hamkins2013.MO144236:IsTheInclusionVersionOfTheKunenInconsistencyTheoremTrue?}]
\label{Corollary.NontrivialAutomorphismsOfOf}
 Set-theoretic mereology is not rigid. That is, in every model of set theory $\<V,{\in}>$, there are numerous nontrivial definable automorphisms of the inclusion relation $\tau:\<V,\of>\cong\<V,\of>$.
\end{corollary}

\begin{proof}
 This is precisely what the construction of the map $\tau$ in the proof of theorem~\ref{Theorem.DifferentInSameOf} provides. Note that distinct choices of $\theta$ lead to distinct such~$\of$-automorphisms $\tau$.
\end{proof}

Another way to express what is going on in the proof is that $\tau$ is an isomorphism of the structure $\< V,{\in^*},{\subseteq}>$ with $\<V,{\in},{\subseteq}>$, and so $\subseteq$ is in fact the same as the corresponding inclusion relation $\subseteq^*$ that one would define from $\in^*$. Corollary~\ref{Corollary.NontrivialAutomorphismsOfOf} contrasts with the fact that \ZFC\ proves that $\<V,{\in}>$ and indeed any transitive set or class is rigid, since if $\pi:V\to V$ is an~$\in$-respecting bijection, there can be no~$\in$-minimal set $a$ with $\pi(a)\neq a$.

\begin{corollary}\label{Corollary.InNotDefinableFromOf}
 One cannot define~$\in$ from $\subseteq$ in any model of set theory, even allowing parameters in the definition.
\end{corollary}

\begin{proof}
For any parameter $z$, let us choose $z$-definably the bijection $\theta$ in the proof of theorem~\ref{Theorem.DifferentInSameOf} to be nontrivial, while still having $\theta\image z=z$. For example, perhaps $\theta$ swaps $z$ and $\{z\}$, and leaves all other sets unchanged. From this, it follows by the proof of theorem~\ref{Theorem.DifferentInSameOf} that the map $\tau:a\mapsto\theta\image a$ is an~$\of$-automorphism, and our choice of $\theta$ ensures that $\tau(z)=z$. So $\tau$ preserves every relation definable from $\subseteq$ and parameter $z$. But $\tau$ does not preserve~$\in$, and consequently,~$\in$ must not be definable from~$\of$ using parameter $z$.
\end{proof}

Nevertheless, for a counterpoint, we claim that there is a weak sense in which the isomorphism type of $\< V,\in>$ is implicit in the inclusion relation $\subseteq$, namely, any other class relation $\in^*$ having that same inclusion relation is isomorphic to the~$\in$ relation.

\begin{theorem}\label{Theorem.IfInclusionSameThenInIsIsomorphic}
Assume ZFC in the universe $\<V,\in>$. Suppose that $\in^*$ is a definable class relation in $\<V,{\in}>$ for which $\<V,\in^*>$ is a model of set theory (a weak set theory suffices), such that the corresponding inclusion relation $$u\subseteq^* v\quad\iff\quad\forall a\,(a\in^* u\to a\in^* v)$$is the same as the usual inclusion relation $u\subseteq v$. Then the two membership relations are isomorphic $$\<V,\in>\cong\<V,\in^*>.$$
\end{theorem}

\begin{proof}
Since a singleton set $\{a\}$ has exactly two subsets with respect to the usual $\subseteq$ relation---the empty set and itself---this must also be true with respect to the inclusion relation $\subseteq^*$ defined via $\in^*$, since we have assumed $\subseteq^*=\subseteq$. Since only singletons have exactly two subsets, the object $\{a\}$ must also be a singleton with respect to $\in^*$, and consequently there is a unique object $\eta(a)$ such that $x\in^*\singleton{a}\iff x=\eta(a)$. Since every object has a singleton with respect to $\in^*$, it follows that $\eta$ is surjective, and since every object has a unique singleton with respect to $\in^*$, it follows that $\eta$ is injective. So $\eta:V\to V$ is bijective. Let $\theta=\eta^{-1}$ be the inverse permutation.

Unwrapping things, we may observe that $$a\in u\quad\iff\quad \{a\}\subseteq u\quad\iff\quad \{a\}\subseteq^* u\quad\iff\quad\eta(a)\in^* u,$$ and so $a\in u\iff\eta(a)\in^* u$. By taking inverses, we deduce for any sets $b$ and $u$ that
 $$b\in^* u\quad\iff\quad \theta(b)\in u.$$
Using~$\in$-recursion, let us define
 $$b^*=\{\ \theta(a^*)\mid a\in b\ \}.$$
This is the step of the proof where we need $\in^*$ to be a definable class with respect to~$\in$, in order that the class function $\eta$ and hence also $\theta$ are classes with respect to~$\in$, so that we may legitimately undertake~$\in$-recursion using them. We do not actually need that $\in^*$ is {\em definable} in $\<V,{\in}>$, but rather we need that $\in^*$ is an amenable class to this structure, in the sense that the axioms of \ZFC\ hold even if we allow a predicate for $\in^*$. For example, the argument would work in \Godel-Bernays set theory, provided that $\in^*$ is a class, as we state in corollary~\ref{Corollary.GBCSameOfIsomorphicIn}.

Continuing with the proof, we claim next by~$\in$-induction that the map $b\mapsto b^*$ is one-to-one, since if there is no violation of this for the elements of $b$, then we may recover $b$ from $b^*$ by applying $\theta^{-1}$ to the elements of $b^*$ and then using the induction assumption to find the unique $a$ from $a^*$ for each $\theta(a^*)\in b^*$, thereby recovering $b$. So $b\mapsto b^*$ is injective.

We claim that this map is also surjective. If $y_0\neq b^*$ for any $b$, then there must be an element of $y_0$ that is not of the form $\theta(a^*)$ for any $a$, since otherwise we would be able to realize $y_0$ as the corresponding $b^*$. Since $\theta$ is surjective, this means there is $\theta(y_1)\in y_0$ with $y_1\neq b^*$ for any $b$. Continuing, there is $y_{n+1}$ with $\theta(y_{n+1})\in y_n$ and $y_{n+1}\neq b^*$ for any $b$. Let $z=\{\ \theta(y_n)\mid n\in\omega\ \}$. Since $x\in^* u\iff \theta(x)\in u$, it follows that the $\in^*$-elements of $z$ are precisely the $y_n$'s. But $\theta(y_{n+1})\in y_n$, and so $y_{n+1}\in^* y_n$. So $z$ has no $\in^*$-minimal element, violating the axiom of foundation for $\in^*$, a contradiction. So the map $b\mapsto b^*$ is a bijection of $V$ with $V$.

Finally, we observe that because $$a\in b\iff\theta(a^*)\in b^*\iff a^*\in^* b^*,$$ it follows that the map $b\mapsto b^*$ is an isomorphism of $\<V,\in>$ with $\<V,\in^*>$, as desired.
\end{proof}

The conclusion is that although~$\in$ is not definable from $\subseteq$, nevertheless, the isomorphism type of~$\in$ is implicit in $\subseteq$, in the weak sense that any other class relation $\in^*$ giving rise to the same inclusion relation $\subseteq^*=\subseteq$ is isomorphic to~$\in$. The proof actually shows the following:

\begin{corollary}\label{Corollary.GBCSameOfIsomorphicIn}
 In \Godel-Bernays set theory, if $\in^*$ is a class binary relation and $\<V,{\in^*}>$ happens to be a model of set theory and has the same inclusion relation $\of^*=\of$ as the usual inclusion relation $\of$ defined from~$\in$, then $\<V,{\in^*}>$ is isomorphic to $\<V,{\in}>$.
\end{corollary}

As we have mentioned, the argument uses that $\in^*$ is a class (so that we have \ZFC\ in the language of the structure $\<V,{\in},{\in^*}>$), and it is natural to wonder whether one can omit that hypothesis. For example, perhaps $\in^*$ is not definable in $\<V,{\in}>$ nor even amenable to this structure. It would be a much stronger result with philosophical significance to show that~$\of$ can truly identify the isomorphism class of~$\in$.

\begin{question}\label{Question.NonIsomorphicInButSameSubset}
 Can there be two models of set theory $\<W,{\in}>$ and $\<W,{\in^*}>$, not necessarily classes with respect to each other in the sense of \Godel-Bernays set theory, which have the same underlying universe $W$ and the same inclusion relation $\subseteq=\subseteq^*$, but which are not isomorphic?
\end{question}

For example, can we arrange that $\<W,{\in}>$ has the continuum hypothesis and $\<W,{\in^*}>$ does not? In theorem~\ref{Theorem.DifferentTheoriesSameOf}, we prove that the answer is affirmative. In fact, one can arrange that the models have any desired consistent theories extending \ZFC, but with the same inclusion relation. This result shows that, contrary to what might have been suggested by theorem~\ref{Theorem.IfInclusionSameThenInIsIsomorphic}, the inclusion relation~$\of$ does not actually identify the~$\in$-isomorphism class of the universe, or even the~$\in$-theory of the model of set theory in which it arises.

\section{Set-theoretic mereology is a decidable theory}

Let us turn now to the fact that set-theoretic mereology, considered as the theory of the structure $\<V,{\of}>$, constitutes a decidable theory. This, on our view, appears to be devastating for any attempt to use set-theoretic mereology by itself as a foundation of mathematics, a view we shall discuss further in section \S\ref{Section.Conclusions}. Meanwhile, while proving the decidability result in this section, we shall also identify exactly what is the complete theory of $\of$-based set-theoretic mereology.

To warm up, consider first the easier case of the set $\HF$ of hereditarily finite sets, which in terms of the von Neumann hierarchy is the same as $\HF=V_\omega$. Note that $\<\HF,{\of}>$ is a lattice order. Furthermore, every element of $\HF$ is a finite subset of $\HF$, a countable set, and every such finite subset is realized in $\HF$. Hence, $\<\HF,{\of}>$ is simply isomorphic to the lattice of finite subsets of a countable set, such as the collection of finite subsets of $\N$ under inclusion. Such a lattice structure is well-known to be decidable. This lattice is isomorphic, for example, to the lattice of square-free natural numbers under divisibility, associating each square-free number with the set of its prime divisors, and in this sense, inclusion in \HF\ is analogous to divisibility in arithmetic. The lattice of divisibility for square-free numbers, in turn, is definable in the natural numbers $\<\N,{\mid}>$ under divisibility $\mid$, a structure also known to be decidable; divisibility is definable from multiplication, and $\<\N,{\cdot}>$ also is decidable. So  $\<\HF,{\of}>$ has a decidable theory.

Let us now consider more generally the full structure $\<V,\of>$, where $V$ is the entire set-theoretic universe. We shall prove that this structure also has a decidable theory, and indeed it will follow from our analysis that it has exactly the same theory as $\<\HF,{\of}>$.

\begin{theorem}\label{Theorem.MereologyIsDecidable}
 Set-theoretic mereology, considered as the theory of $\<V,\of>$, is precisely the theory of an atomic unbounded relatively complemented distributive lattice, and furthermore, this theory is finitely axiomatizable, complete and decidable.
\end{theorem}

We shall prove this theorem by means of the more specific quantifier-elimination argument of theorem~\ref{Theorem.QuantifierElimination}. These results should be viewed as partaking in Tarski's classification of the elementary classes of Boolean algebras by means of the Tarski invariants (see~\cite[theorem~5.5.10]{ChangKeisler1990:ModelTheory}) and \Ersov's extension of that work to the case of relatively complemented distributive lattices~\cite{Ershov1964:DecidabilityOfTheElementaryTheoryOfRelativelyComplementedLatticesAndOfTheTheoryOfFilters}. Tarski had used a quantifier-elimination argument to show the decidability of atomless Boolean algebras, which led him to the Tarski invariants for all Boolean algebras. \Ersov's generalization proceeds by a general technique (described in~\cite[theoremf~15.6]{Monk1976:MathematicalLogic} and also~\cite[theorem~3.1.1]{BaudischSeeseTuschikWeese1985:DecidabilityAndQuantifierElimination}) allowing him to establish a large number of decidability results. For example, \Ersov's method enabled him to handle a Boolean algebra expanded by a predicate for a prime ideal (see~\cite[p.~1054]{Weese1989:DecidableExtensionsOfTheTheoryOfBooleanAlgebras}). The statement and proof of theorem~\ref{Theorem.QuantifierElimination} has an affinity with the corresponding quantifier-elimination result for infinite atomic Boolean algebras, as in \cite[theorem~6.20]{Poizat2000:ACourseInModelTheory}; see also~\cite[p.~66]{Hodges1993:ModelTheory}. So the decidability result we are claiming here is not new and follows immediately from \Ersov's proof that the theory of relatively complemented distributive lattices is decidable (regardless of atomicity and unboundedness), and the quantifier-elimination result is similar to that for infinite atomic Boolean algebras. Nevertheless, in order to provide a self-contained presentation, let us give here a direct elimination-of-quantifiers argument for the central case of set-theoretic mereology: an atomic unbounded relatively complemented distributive lattice.

\begin{theorem}\label{Theorem.QuantifierElimination}
 If $\<W,\in^W>$ is a model of set theory with the corresponding inclusion relation~$\of$, then $\<W,\of>$ is an atomic unbounded relatively complemented distributive lattice, and this theory satisfies the elimination of quantifiers in the language containing the Boolean operations of intersection $x\intersect y$, union $x\union y$, relative complement $x-y$ and the unary size relations $|x|=n$, for each natural number $n$.
\end{theorem}

\begin{proof}
In any model of set theory, the subset relation~$\of$ is of course a partial order and indeed a lattice order, since any two sets $a$ and $b$ have a least upper bound, the union $a\union b$, and a greatest lower bound, the intersection $a\intersect b$. The lattice is distributive, because intersection and union both distribute over the other. The lattice has a least element $\emptyset$, but no greatest element (and this is what we mean by unbounded; for a relatively complemented distributive lattice, it is equivalent to saying that it is not a Boolean algebra). The lattice is relatively complemented, since for any two sets $a,b$, the difference set $a-b$ is the complement of $b$ relative to $a$, meaning that $b\intersect (a-b)=0$ and $a=(a\intersect b)\union(a-b)$. The lattice is atomic, since every nonempty set is the join of the singleton sets below it, and those singletons are atoms with respect to inclusion. In summary, $\<W,\of>$ is an atomic unbounded relatively complemented distributive lattice, and this is what we shall use for the quantifier-elimination argument.

In any lattice, for any natural number $n$ we may introduce a unary predicate, which we shall write as $|x|=n$, which we define to hold precisely when $x$ is the join of $n$ distinct atoms. For any particular $n$, this relation is expressible in the language of lattices, that is, from~$\of$ in our case. In our model of set theory, this relation expresses that $x$ is a finite set with $n$ elements. Similarly, in any lattice let us introduce the unary predicate denoted $|x|\geq n$, which expresses that $x$ admits a decomposition as the join of $n$ distinct nonzero incompatible elements: $x=y_1\union\cdots\union y_n$, where $y_i\neq 0$ and $y_i\intersect y_j=0$ for $i\neq j$. In an atomic relatively complemented lattice, the relation $|x|\geq n$ holds just in case there are at least $n$ atoms $a\leq x$. This relation is also definable from the lattice order.

We shall prove that every formula in the language of lattices is equivalent, over the theory of atomic unbounded relatively complemented distributive lattices, to a quantifier-free formula in the language of the order $a\of b$, equality $a=b$, meet $a\intersect b$, join $a\union b$, relative complement $a-b$, constant $0$, the unary relations $|x|=n$ and $|x|\geq n$, where $n$ is respectively any natural number.

We prove the result by induction on formulas. The collection of formulas equivalent to a quantifier-free formula in that language clearly includes all atomic formulas and is closed under Boolean combinations. So it suffices to eliminate the quantifier in a formula of the form $\exists x\, \varphi(x,\ldots)$, where $\varphi(x,\ldots)$ is quantifier-free in that language. Let us make a number of observations that will enable various simplifying assumptions about the form of $\varphi$.

Because equality of terms is expressible by the identity $a=b\iff a\of b\of a$, we do not actually need $=$ in the language (and here we refer to the use of equality in atomic formulas of the form $s=t$ where $s$ and $t$ are terms, and not to the incidental appearance of the symbol $=$ in the unary predicate $|x|=n$). Similarly, in light of the equivalence $a\of b\iff |a-b|=0$, we do not need to make explicit reference to the order $a\of b$. So we may assume that all atomic assertions in $\varphi$ have the form $|t|\geq n$ or $|t|=n$ for some term $t$ in the language of meet, join, relative complement and $0$. We may omit the need for explicit negation in the formula by systematically applying the equivalences:
$$\neg(|t|\geq n)\iff \bigvee_{k<n}|t|=k\quad\text{ and}$$
$$\neg(|t|=n)\iff (|t|\geq n+1)\vee\bigvee_{k<n}|t|=k.$$
So we have reduced to the case where $\varphi$ is a positive Boolean combination of expressions of the form $|t|\geq n$ and $|t|=n$.

Let us consider the form of the terms $t$ that may arise in the formula. List all the variables $x_0,x_1,\ldots,x_N$ that arise in any of the terms appearing in $\varphi$, where $x_0$ is the variable $x$, and consider the Venn diagram corresponding to these variables. The cells of this Venn diagram can each be described by a term of the form $\Intersect_{i\leq N} \pm x_i$, which we shall refer to as a cell term, where $\pm x_i$ means that either $x_i$ appears or else we have subtracted $x_i$ from the other variables. For example, $(x_0\intersect x_3)-(x_1\union x_2\union x_5)$ is a cell term in five variables, describing a cell of the corresponding Venn diagram. Since we have only relative complements and not absolute complements, we need only consider the cells where at least one variable appears positively, since the exterior region in the Venn diagram is not actually represented by any term. In this way, every term in the language of relatively complemented lattices is equivalent to a term that is a finite union of such cell terms, plus $\emptyset$ (which could be viewed as an empty union). Note that distinct cell terms are definitely representing disjoint objects in the lattice.

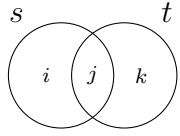
\begin{center}
\begin{tikzpicture}[scale=.7]
\draw (0,0) circle (1cm) node[left] {\tiny $i$};
\draw (-.5,.8) node[above left] {$s$};
\draw (1.2,0) circle (1cm) node[right] {\tiny $k$};
\draw (1.7,.8) node[above right] {$t$};
\draw (.6,0) node {\tiny $j$};
\end{tikzpicture}
\end{center}

Next, by considering the possible sizes of $s-t$, $s\intersect t$ and $t-s$ as illustrated in the diagram above, we may observe the identities
 $$|s\union t|=n\ \iff\bigvee_{i+j+k=n}(|s|=i+j)\wedge(|s\intersect t|=j)\wedge(|t|=j+k)\phantom{,}$$
 $$|s\union t|\ \geq n\iff\bigvee_{i+j+k=n}(|s|\geq i+j)\wedge(|s\intersect t|\geq j)\wedge(|t|\geq j+k).$$
Through repeated application of this, we may reduce any size assertion about any term $t$ to a Boolean combination of assertions about cell terms. (Note that size assertions about $\emptyset$ are trivially settled by the theory and can be eliminated.)

Let us now focus on the quantified variable $x$ separately from the other variables, for it may appear either positively or negatively in such a cell term. More precisely, each cell term in the variables $x,x_1,\ldots,x_N$ is equivalent to $x\intersect c$ or $c-x$, for some cell term $c$ in the variables $x_1,\ldots,x_N$, that is, not including $x$, or to the term $x-(x_1\union\cdots\union x_N)$, which is the cell term for which $x$ is the only positive variable.

We have reduced the problem to the case where we want to eliminate the quantifier from $\exists x\, \varphi$, where $\varphi$ is a positive Boolean combination of size assertions about cell terms. We may express $\varphi$ in disjunctive normal form and then distribute the quantifier over the disjunct to reduce to the case where $\varphi$ is a conjunction of size assertions about cell terms. Each cell term has the form $x\intersect c$ or $c-x$ or $x-(x_1\union\cdots \union x_N)$, where $c$ is a cell term in the list of variables without $x$. Group the conjuncts of $\varphi$ that use the same cell term $c$ in this way together. The point now is that assertions about whether there is an object $x$ in the lattice such that certain cell terms obey various size requirements amount to the conjunction of various size requirements about cells in the variables not including $x$. For example, the assertion $$\exists x\,(|x\intersect c|\geq 3)\wedge(|x\intersect c|\geq 7)\wedge(|c-x|=2)$$ is equivalent (over the theory of atomic unbounded relatively complemented distributive lattices) to the assertion $|c|\geq 9$, since we may simply let $x$ be all but $2$ atoms of $c$, and this will have size at least $7$, which is also at least $3$, and the complement $c-x$ will have size $2$. If contradictory assertions are made, such as $\exists x\, (|x\intersect c|\geq 5\wedge |x\intersect c|=3)$, then the whole formula is equivalent to $\perp$, which can be expressed without quantifiers as $0\neq 0$.

Next, the key observation of the proof is that positive assertions about the existence of such $x$ for different cell terms in the variables not including $x$ will succeed or fail independently, since those cell terms are representing disjoint elements of the lattice, and so one may take the final witnessing $x$ to be the union of the witnesses for each piece. So to eliminate the quantifier, we simply group together the atomic assertions being made about the cell terms in the variables without $x$, and then express the existence assertion as a size requirement on those cell terms. For example, the assertion $$\exists x\, (|c\intersect x|\geq 5)\wedge(|c-x|=6)\wedge (|d\intersect x|\geq 7),$$ where $c$ and $d$ are distinct cell terms in the other variables, is equivalent to $$(|c|\geq 11)\wedge(|d|\geq 7),$$ since $c$ and $d$ are disjoint and so we may let $x$ be the appropriate part of $c$ and a suitable piece of $d$. The only remaining complication concerns instances of the term $x-(x_1\union\cdots\union x_N)$. But for these, the thing to notice is that any single positive size assertion about this term is realizable in our theory, since we have assumed that the lattice is unbounded, and so there will always be as many atoms as desired disjoint from any finite list of objects. But we must again pay attention to whether the requirements expressed by distinct clauses are contradictory.

Altogether, this provides a procedure for eliminating quantifiers from any assertion in the language of lattices down to the language augmented by unary predicates expressing the size of an object. This procedure works in any atomic unbounded relatively complemented distributive lattice, and so the theorem is proved.
\end{proof}

\begin{corollary}\label{Corollary.Complete}
 The theory of atomic unbounded relatively complemented distributive lattices is complete and decidable.
\end{corollary}

\begin{proof}
Theorem~\ref{Theorem.QuantifierElimination} shows that every sentence in this theory is equivalent to a quantifier-free sentence in the expanded language with the unary size predictes. But since such sentences have no variables, they must simply be a Boolean combination of trivial size assertions about $0$, such as $(|0|\geq 2)\vee \neg(|0|=5)$, and the truth value of all such assertions is settled by the theory. So the theory of atomic unbounded relatively complemented distributive lattices is finitely axiomatizable and complete. Every such theory is decidable: given any sentence, simply search for a proof of it or the negation.
\end{proof}

Theorem~\ref{Theorem.MereologyIsDecidable} follows from theorem~\ref{Theorem.QuantifierElimination} and corollary~\ref{Corollary.Complete}. One can also view the proof of theorem~\ref{Theorem.QuantifierElimination} as providing an explicit decision procedure: given a sentence, use the procedure to find the quantifier-free equivalent, which will be a trivial assertion about $0$, whose truth is easily determined.

The quantifier-elimination result also has the following consequence, which unifies theorem~\ref{Theorem.MereologyIsDecidable} with our remarks about $\HF$.

\begin{corollary}\label{Corollary.HFElementary}
 The structure of hereditarily finite sets $\HF$ under inclusion is an elementary substructure of the entire set-theoretic universe $V$ under inclusion. $$\<\HF,{\of}>\elesub\<V,{\of}>$$
\end{corollary}

\begin{proof}
These structures are both atomic unbounded relatively complemented distributive lattices, and so they each support the quantifier-elimination procedure. But they agree on the truth of any quantifier-free assertion about the sizes of hereditarily finite sets, and so they agree on all truth assertions about objects in $\HF$ in the language of~$\of$.
\end{proof}

\section{Mereology does not identify~$\in$ up to isomorphism}

We would like to tie up a loose end from our presentation of theorem~\ref{Theorem.IfInclusionSameThenInIsIsomorphic}, which identified a weak sense in which we are able to define the isomorphism class of~$\in$ from~$\of$. We had left it unsettled in question~\ref{Question.NonIsomorphicInButSameSubset} whether this weak sense could hold more robustly. We should like now to prove that in fact it does not. The following theorem shows that $\of$-based set-theoretic mereology is unable to distinguish the~$\in$-theory of the model of set theory in which it arises, and so~$\of$ cannot truly identify the~$\in$-isomorphism class of the model in which it resides.

\begin{theorem}\label{Theorem.DifferentTheoriesSameOf}
 For any two consistent theories extending $\ZFC$, there are models $\<W,{\in}>$ and $\<W,{\in^*}>$ of those theories, respectively, with the same underlying set $W$ and the same induced inclusion relation $\of=\of^*$.
\end{theorem}

\begin{proof}
Suppose that $T$ and $T^*$ are two consistent theories extending \ZFC\ in the language of set theory. Note that by the \Lowenheim-Skolem theorem, if there are models as stated in the conclusion of the theorem, then there are countable models like that, and so the conclusion of the theorem has complexity $\Sigma^1_1$ in descriptive set theory. By the Shoenfield absoluteness theorem, therefore, it follows that the conclusion of the theorem is absolute to every forcing extension, including forcing extensions where the continuum hypothesis holds. We may therefore assume, without loss of generality, that the continuum hypothesis holds. It follows that there are models $\<W,{\in}>$ and $\<W,\hat\in>$ of $T$ and $T^*$, respectively, which are countably saturated on a common domain $W$ of size $\aleph_1$. That is, for each of the models and for any countable list of formulas $\varphi_i(x)$ using countably many parameters from $W$, if every finite subcollection of the formulas is realized in the model, then the whole collection is realized, meaning that some $a$ in the model satisfies every $\varphi_i(a)$. It follows that the corresponding defined structures $\<W,\of>$ and $\<W,\hat\of>$ are also both saturated. By theorem~\ref{Theorem.MereologyIsDecidable}, these are both models of the theory of infinite atomic relatively complemented distributive lattices with no largest element, and this is a complete theory. In particular, these two structures are elementarily equivalent and saturated. It follows by the usual back-and-forth construction that there is an isomorphism $\pi:\<W,\of>\cong\<W,\hat\of>$. To construct $\pi$, simply enumerate the elements of $W$ as $\<a_\alpha\mid\alpha<\omega_1>$ and define $\pi$ in stages. At stage $\alpha$, consider the type of $a_\alpha$ in $\<W,\of>$ using parameters $a_\beta$ for $\beta<\alpha$, and find an element $\pi(a_\alpha)$ that realizes the same type over $\<W,\hat\of>$ using parameters $\pi(a_\beta)$; this type is finitely realizable since each instance of this was part of the earlier types, and therefore it is realized by saturation; we can similarly ensure that $\pi$ is surjective, and so it is an isomorphism $\pi:\<W,\of>\cong\<W,\hat\of>$, as desired. Given the isomorphism, define $a\in^*b$ just in case $\pi(a)\mathrel{\hat\in}\pi(b)$, so that $\pi:\<W,{\in^*}>\cong\<W,\hat\in>$, and so $\<W,{\in^*}>\satisfies T^*$. Observe that
 $$u\of^* v\quad\iff\quad \pi(u)\mathrel{\hat\of}\pi(v)\quad\iff\quad u\of v,$$
where the first equivalence follows from the fact that $\pi$ is an isomorphism of $\<W,{\in^*}>$ with $\<W,\hat\in>$, and the second follows from the fact that $\pi$ is an isomorphism of $\<W,{\of}>$ with $\<W,\hat\of>$. So $\<W,{\in}>$ and $\<W,{\in^*}>$ are models of $T$ and $T^*$, respectively, but the corresponding defined inclusion relations are identical $\of=\of^*$.
\end{proof}

As a consequence, set-theoretic mereology appears to be oblivious to the independence phenomenon in set theory, otherwise widespread in set theory, in that it fails to distinguish between models with extremely different theories in the usual language of set theory with~$\in$, because these models can be identical in the language with only~$\of$. In particular, contrary to what might have been taken as the suggestion of theorem~\ref{Theorem.IfInclusionSameThenInIsIsomorphic}, if we are given the inclusion relation~$\of$ of a model of set theory, we cannot generally identify the isomorphism class of the~$\in$ relation from which it arose, or even the~$\in$-theory of that structure. Given only~$\of$, we cannot determine whether the continuum hypothesis holds or fails or whether there are large cardinals or not (or indeed even whether there are infinite sets or not).

\section{Mereology with the singleton operator}\label{Section.MereologyWithSingltonOperator}

Until now, we have undertaken what might be described as a study of a {\it pure} set-theoretic mereology, where we have only the inclusion relation~$\of$. But it might be natural to consider an expanded mereology, where we augment the inclusion relation~$\of$ by allowing reference also to other kinds of set-theoretic structure. For example, let us now consider the theory that arises when we augment mereology by adding the singleton operator $s:a\mapsto \singleton{a}$, which maps every object to its own singleton.

In the context of mereology, this is not as innocent as it may appear to a contemporary set theorist. Indeed, there is a surprisingly rich history of confusion and controversy about the singleton concept stretching back into the earliest days of set theory and pre-set-theory (see~\cite{Kanamori2003:TheEmptySetTheSingletonAndTheOrderedPair}), and the concept is controversial in connection with mereology, from the beginning of the subject. More recently, a literature of criticism of the singleton has arisen in response to David Lewis's~\cite{Lewis1991:PartsOfClasses} development of a class-based set-theoretic mereology. See~\cite{Hellman2009:MereologyInPhilosophyOfMathematics} and~\cite{ChampollionKrifka:Mereology} for further discussion of the singleton and mereological atoms.

What we should like to observe here is merely that if we were to expand the language by adding the singleton operator $s$ as well as the inclusion relation $\of$, then we would get a structure that is equally as powerful as the usual membership-based set theory.

\begin{theorem}\label{Theorem.MereologyWithSingletonOperator}
 Every model of membership-based set theory $\<V,\in>$ is interdefinable with the corresponding singleton-expanded mereological model $\<V,\of,s>$.
\end{theorem}

\begin{proof}
For the one direction, we can easily define~$\of$ and the singleton operator $s$ using~$\in$ as follows:
 $$u\of v\quad\iff\quad\forall x\, (x\in u\implies x\in v)$$
 $$y=s(x)\quad\iff\quad \forall z\, (z\in y\iff z=x).$$
Conversely, we may define~$\in$ from~$\of$ and $s$ via
 $$x\in y\quad\iff\quad s(x)\of y.$$
So the theorem is proved.
\end{proof}

Thus, this stronger version of mereology, expanded by the singleton operator, is basically equivalent to membership-based set theory as far as the foundations of mathematics is concerned. One could express, for example, that the~$\in$-universe $\<V,\in>$ satisfied \ZFC\ plus certain large cardinals, by mereological assertions expressed purely in the language of inclusion and the singleton operator. One would want to assert that $\of$ is an atomic unbounded relatively complemented distributive lattice, that $s$ is a bijection of the universe with the $\of$-atoms, but also that additional properties hold that ensure the axioms of infinity, separation, power set, replacement and so on.

\section{Conclusions}\label{Section.Conclusions}

Let us now discuss our philosophical conclusions concerning the suitability of the~$\of$-based interpretation of set-theoretic mereology to serve as a foundation of mathematics. These conclusions grow naturally out of the mathematical ideas we have presented in the earlier sections of this article.

Consider first theorem~\ref{Theorem.DifferentInSameOf} and corollary~\ref{Corollary.InNotDefinableFromOf}, which show that~$\in$ is not generally definable from~$\of$ in models of set theory. Although this shows that~$\in$-based set theory is not founded upon~$\of$-based set theory in a superficial, direct manner, we do not actually take these results to rule out~$\of$-based set-theoretic mereology as a foundation of mathematics. Rather, the results merely close off what might have been naively hoped for as an easy way to establish mereology as a powerful foundation, namely, the idea that perhaps~$\in$ and~$\of$ were interdefinable. They are not interdefinable, as the theorem and corollary establish, and so the easy road is blocked. But in order for~$\of$ to serve as a foundation of mathematics, it is not required for these relations to be interdefinable. Rather, all that would be required is that we should be able to find faithful representations of all our other mathematical structures, such as~$\in$-based set theory, within the~$\of$-based set-theoretic mereology. To insist that~$\in$ is definable from~$\of$ would be to insist further that the way that~$\of$ serves as a foundation is exactly the inverse of the way that~$\in$ happens to serve easily as a foundation for~$\of$. But perhaps~$\of$ might serve as a foundation for~$\in$-based set theory in some other way; perhaps~$\in$-based set theory is simulated within~$\of$-based set theory by means of a much more complicated interpretation or structure. (In the end, we don't believe so, but not solely on the basis of theorem~\ref{Theorem.DifferentInSameOf} and corollary~\ref{Corollary.InNotDefinableFromOf}, which do not seem to rule this out.)

Indeed, hope for such a more complicated but successful interpretation of~$\in$ within~$\of$ might have been buoyed up by theorem~\ref{Theorem.IfInclusionSameThenInIsIsomorphic} and corollary~\ref{Corollary.GBCSameOfIsomorphicIn}, which seem to suggest that perhaps~$\of$-based set-theoretic mereology might be able to identify the isomorphism class of the~$\in$ relation. After all, many other non-set-theory-based foundations of mathematics, such as those originating in category theory, weave the philosophy of mathematical structuralism into the foundational theory, and for these theories it is emphasized that one shouldn't necessarily be able to define a mathematical structure such as~$\in$ directly, but rather merely identify mathematical structure up to isomorphism. Theorem~\ref{Theorem.IfInclusionSameThenInIsIsomorphic} shows that any two class relations~$\in$ and $\in^*$ with the same~$\of$ relation, provided that $\in^*$ is a class with respect to~$\in$, are isomorphic, and thus as in corollary~\ref{Corollary.GBCSameOfIsomorphicIn}, it is correct to say as an internal matter of \Godel-Bernays set theory, that~$\in$ is up to isomorphism the only class relation that forms a model of set theory and gives rise to the actual~$\of$ relation. This is a sense in which~$\of$ knows about~$\in$ up to isomorphism.

But our more considered view is that this is not the same as saying that~$\of$ determines the~$\in$-isomorphism class of the universe, and the situation is clarified by theorem~\ref{Theorem.DifferentTheoriesSameOf}, resolving question~\ref{Question.NonIsomorphicInButSameSubset}. The fact of the matter is that knowledge of the inclusion relation~$\of$ tells you almost nothing about the~$\in$-isomorphism class of the universe in which it arises. Indeed, theorem~\ref{Theorem.DifferentTheoriesSameOf} shows you can learn very little even about the~$\in$-theory of the universe from considering only~$\of$, since any two consistent set theories can have respective models on the same universe of sets with the same~$\of$ relation, even if otherwise they are extremely different on set-theoretic matters. For example, the continuum hypothesis can hold in one model and fail in another, even when those models have exactly the same objects and the same~$\of$ relation (but different~$\in$ relations). Similarly, the models can disagree on other set-theoretic issues, and corollary~\ref{Corollary.HFElementary} shows even that one cannot tell if there are infinite sets or not by looking only at the theory of~$\of$. At bottom, the conclusion seems inescapable that the inclusion relation~$\of$ knows very little set theory.

Nevertheless, what we've said so far does not actually seem decisively to rule out $\of$-based set theory as a foundation, because there is no requirement that~$\of$ need to capture the actual~$\in$-truth of a model of set theory in which it might arise. Rather, in foundations we can be free to find some other faithful simulation of mathematical structure.

So finally, let us come to theorem~\ref{Theorem.MereologyIsDecidable}, in which we notice that~$\of$-based set-theoretic mereology is an atomic unbounded relatively complemented distributive lattice and prove that this is a finitely axiomatizable, complete theory, which is therefore also decidable. This, on our view, is devastating for this formulation of mereology as foundational. We base this view on the following principle, which we should like now to discuss in further detail.

\newtheorem*{principle}{Nondecidability requirement}
\begin{principle}
 If a theory is decidable, then it cannot serve as a foundation of mathematics.
\end{principle}

If this principle is correct, then since theorem~\ref{Theorem.MereologyIsDecidable} shows that the particular formulation of set-theoretic mereology via the inclusion relation~$\of$ is a decidable theory, we assert that this formulation of mereology cannot serve as a foundation of mathematics. And a similar argument will apply to other formulations of mereology, if they lead to a decidable theory. For example, Lewis~\cite{Lewis1991:PartsOfClasses} considers a proper-class-based formulation of set-theoretic mereology that appears to result in an infinite atomic Boolean algebra, which by Tarski's analysis has a decidable theory (as in ~\cite[theorem~6.20]{Poizat2000:ACourseInModelTheory}, which is the Boolean-algebra analogue of theorem~\ref{Theorem.QuantifierElimination}). So using only the $\of$ relation in that formulation of set-theoretic mereology would seem similarly to be inadequate as a foundation. (But Lewis also considers singletons, and in light of theorem~\ref{Theorem.MereologyWithSingletonOperator}, allowing the singleton operator would recover $\in$ and therefore be foundationally robust.)

So let us conclude this paper by discussing the grounds that one might have for the non-decidability requirement. Part of what it means for a theory to be foundational is that one might find faithful representations of the principal mathematical structures within that theory. For example,~$\in$-based set theory is commonly taken as a possible foundation of mathematics, because set theorists have observed that one can find within set theory seemingly faithful copies of all the usual mathematical structures considered in mathematics, and we can formalize analogues of all the usual mathematical procedures that mathematicians might employ in connection with those structures. In set theory, we have a way of talking about ordered pairs and functions and relations and orders and we can build a copy of the natural numbers and the rational field and we can build a complete ordered field and so on. For a theory to be acceptable as a foundation of mathematics, it must similarly be able to find faithful representations for all the usual mathematical structures. In particular, it must have a way of representing the natural numbers, so that for any particular natural number $n$, we would have a corresponding way of referring to the number $n$ in the theory, and a way of moving from the representation of $n$ to that of $n+1$ and so on with the other arithmetic structure. Similarly, we would have representations of finite combinatorial objects, including the operation, say, of Turing machines.

Suppose that we have a theory $T$ that is foundational in the sense described in the previous paragraph, but also decidable, so that we have a decision procedure for determining whether a given statement is provable from $T$ or not. We do not assume that the theory $T$ is sound, although the representation of arithmetic will mean that it is $\Delta_0$-sound. Working in the meta-theory, where the decision procedure for $T$ exists, let $A$ and $B$ be a computably inseparable pair of computably enumerable sets. (For this argument, kindly allow us to assume that the meta-theory is sufficiently strong to deal with computable enumerations and the existence of such sets, or rather, the existence of the programs enumerating them.) So we have particular programs $p_A$ and $p_B$ that enumerate the elements of $A$ and $B$, respectively; these sets are disjoint; and there is no computable set containing $A$ while disjoint from $B$. But in the meta-theory, consider the set $C$ consisting of those natural numbers $n$ that $T$ proves are enumerated first by $p_A$ (or only by $p_A$) in comparison with $p_B$. Since $T$ is foundational, we are able to express the operation of the Turing machines $p_A$ and $p_B$ in the theory. And furthermore, if a number $n$ is actually in $A$, then it will be enumerated by $p_A$ by some definite computation, which itself would be faithfully represented in $T$, and so $T$ will agree that $n$ is enumerated by $p_A$. Since $A$ and $B$ are disjoint (in the meta-theory), it follows that $n$ will not be enumerated by $p_B$ by any computation of the same length or shorter, and so since the operation of $p_B$ is faithfully represented in $T$ to that same length, it follows that $T$ will agree that $n$ is enumerated first by $p_A$ in comparison with $p_B$. So $n$ will be in $C$, and thus $C$ contains $A$. Similarly, if $n$ is in $B$, then it is enumerated by $p_B$ by some definite computation, which is faithfully represented in $T$, along with all the shorter computations of $p_A$, and so $T$ will agree in this case that $n$ is not first enumerated by $p_A$, and so $n$ will not be in $C$. So $C$ contains $A$ and is disjoint from $B$. Finally, since $T$ is decidable, it follows that $C$ is computable, and so it is a computable separation of $A$ and $B$, contrary to our choice of $A$ and $B$ to be computably inseparable. So this seems to provide grounds for the non-decidability requirement.

One might summarize the argument as the following elementary fact, applied in the meta-theory: if a theory $T$ can formalize arithmetic and is $\Delta_0$-sound, then it cannot be decidable, because if it were, we would be able to find computable separations of computably inseparable sets, which is a contradiction.

For one final remark, let us highlight a subtle aspect of the argument we have just given, particularly the distinction between arguing in the meta-theory in comparison with the object theory, by considering how the analysis works in the confounding case of the theory $T=\ZFC+\neg\Con(\ZFC)$. If \ZFC\ is consistent, then the incompleteness theorem shows that $T$ also is consistent. If we regard \ZFC\ as capable of providing a foundation of mathematics in the sense described in the previous paragraph, then it would seem that $T$ also, being a consistent extension of a foundational theory, would similarly be capable of providing a foundation of mathematics. For if in \ZFC\ we can prove the existence of mathematical objects and structures that faithfully represent our usual mathematical structures, then $T$ being a stronger theory can also prove the existence of these structures and more. Although many set theorists regard $T$ as unsound, because it asserts $\neg\Con(\ZFC)$, which is to say, $T$ asserts the existence of a certain finite combinatorial object, the proof of a contradiction in \ZFC, which we don't expect to find in the meta-theory, nevertheless, this by itself doesn't seem to prevent $T$ from being foundational. One can easily see that $T$ is not a decidable theory, using a computably inseparable pair of computably enumerable sets, as in the previous paragraph. But the confounding thing to notice, here, is that while $T$ is not decidable in the meta-theory, it actually is decidable within the object theory of $T$ itself. That is, externally, we think $T$ is not decidable, but internally, arguing in $T$ itself, we think that \ZFC\ and hence also $T$ is inconsistent and therefore decidable, because $T$ thinks that everything is provable from $T$. Our view of this example is that it is not a counterexample to the non-decidability principle. The theory $T$ is foundational, and not decidable in the meta-theory, even though $T$ itself thinks that $T$ is decidable. Ultimately, we may regard $T$ as capable of serving as a foundation, if an unsound one, because in the meta-theory we do not actually assert $T$ and we recognize that it is consistent although unsound. But a meta-theoretical context in which $T$ is asserted would not be able to regard $T$ as foundational, since it would look upon $T$ as inconsistent.

\bibliographystyle{alpha}
\bibliography{HamkinsBiblio,MathBiblio,WebPosts}

\end{document}